\documentclass[11pt,reqno]{amsart}

\usepackage{latexsym,fullpage}
\usepackage{amssymb}
\usepackage{amsfonts}
\usepackage{amsmath}
\usepackage{amsthm}
\usepackage{color}
\usepackage{graphics,fancybox,pstricks}

\usepackage[latin1]{inputenc}
\usepackage{times}
\usepackage[T1]{fontenc}
\usepackage{pgf,tikz}
\usetikzlibrary{arrows}

\newtheorem{thm}{Theorem}[section]
\newtheorem{lem}[thm]{Lemma}
\newtheorem{cor}[thm]{Corollary}
\newtheorem{pro}[thm]{Proposition}

\theoremstyle{definition}

\newtheorem{exmp}[thm]{Example}

\newtheorem{rem}[thm]{Remark}
\newtheorem{defn-rem}[thm]{Definition-Remark}

\newtheorem{ques}[thm]{Question}

\numberwithin{equation}{section}

\def\M{{\mathbb M}}
\def\X{{\mathbb X}}

\def\H{{\text {\bf H}}}

\def\ra{\longrightarrow}

\def\Z{{\mathbb Z}}
\def\Y{{\mathbb Y}}
\def\P{{\mathbb P}}

\def\ra{\rightarrow}

\def\L{{\mathbb L}}

\def\ds{\displaystyle}

\def\FF{\mathbb F}

\def\depth{{\rm depth}}

\def\k{\Bbbk}

\begin{document}

\title{The Minimal Free Resolution of A Star-Configuration in $\P^n$}

\author[J.P. Park]{Jung Pil Park}
\address{Faculty of Liberal Education, Seoul National University, Seoul, South Korea, 151-837}
\email{jppark@math.snu.ac.kr} %

\author[Y.S. SHIN]{Yong-Su Shin${}^*$}
\address{Department of Mathematics, Sungshin Women's University, Seoul, Korea, 136-742}
\email{ysshin@sungshin.ac.kr }
\thanks{${}^*$This research was supported by a grant from Sungshin Women's University in 2014.}
\thanks{${}^*$Corresponding author}

\begin{abstract}  
We find the minimal free resolution of the ideal of a star-configuration in $\P^n$ of type $(r,s)$ defined by general forms in $R=\k[x_0,x_1,\dots,x_n]$. This generalises the results of \cite{AS:1,GHM} from a specific value of $r=2$ to any value of $1\le r\le n$. Moreover, we show that any star-configuration in $\P^n$ is arithmetically Cohen-Macaulay. As an application, we construct a few of graded Artinian rings, which have the weak Lefschetz property, using the sum of two ideals of star-configurations in $\P^n$. 
\end{abstract}


\date{\today}

\maketitle

\section{Introduction}

Let $R=\k[x_0,x_1,\dots,x_n]$ be an $(n+1)$-variable polynomial ring over an infinite field $\k$ of any characteristic, and $I$ a homogeneous ideal of $R$ (or the ideal of a subscheme in $\P^n$). Then the numerical function
$$
\H(R/I,t):=\dim_\k R_t -\dim_\k I_t
$$
is called a {\em Hilbert function} of the ring $R/I$. If $I:=I_\X$ is the ideal of a subscheme $\X$ in $\P^n$, then we denote the Hilbert function of $\X$ by 
$$
\H_\X(t):=\H(R/I_\X,t).
$$

A star-configuration has been studied to calculate the dimension of the secant variety to a variety of reducible forms in $R$. In \cite{GMS:1}, the authors first introduced it as extremal points sets with maximal Hilbert function and as the support of a family of $\binom{s}{2}$ fat points. Other applications of such star-configurations have been studied in the work of \cite{BH,CGGS, CHT,GHM,S:1,S:2}. In \cite{AS:2}, the authors proved that  
if $F_1,\dots,F_s$ are general forms in $R$ and  
$$
\tilde F_j=\frac{\prod^s_{i=1} F_i}{F_j}\quad \text{for} \quad j=1,\dots,s,
$$
then
$$
\bigcap_{1 \leq i < j \leq s} (F_i, F_j)  = (\tilde F_1,\dots,\tilde F_s).
$$
They called the variety $\X$ in $\P^n$ of the ideal $(\tilde F_1,\dots,\tilde F_s)$  a {\em  star-configuration} in $\P^n$ of type $(2,s)$ defined by general forms $F_1\ldots, F_s$. In this paper, we generalise the definition of a star-configuration in $\P^n$ as follows: Let $F_1,\dots,F_s$ be general forms in $R$ and let $1\le r\le \min\{s,n\}$. Then the variety $\X$ of the ideal 
$$
\bigcap_{1\le i_i\le \cdots \le i_r\le s}(F_{i_1},\dots,F_{i_r}) 
$$
is called {\em a star-configuration in} $\P^n$ of type $(r,s)$ defined by general forms $F_1,\dots,F_s$. Furthermore, if $F_1,\dots,F_s$ are all general {\em linear} forms, then $\X$ is called a {\em linear star-configuration}  in $\P^n$ of type $(r,s)$ (see also \cite{AS:1, AS:2,S:1,S:2}).

Recently, the minimal free resolution for the ideal $I_\X$ of some specific star-configurations were found. More precisely, in \cite{AS:2}, the authors found the minimal free resolution of the ideal of a star-configuration in $\P^n$ of type $(2,s)$. In \cite{GHM}, the authors found the minimal free resolutions of the ideal $I_\X$ of a linear star-configuration $\X$ in $\P^n$ of type $(r,s)$ and the $2$-nd symbolic power of $I_\X$. Moreover, they showed that a linear star-configuration $\X$ in $\P^n$ is arithmetically Cohen-Macaulay ({\em aCM} for short).

 In Section 2 we find the minimal generators of the ideal $I_\X$ of a star-configuration $\X$ in $\P^n$ of type $(r,s)$ defined by $s$ general forms in $R$. In Section 3  we find the minimal free resolution of the ideal $I_\X$ using Eagon-Northcott resolution (see Theorem~\ref{T:20140307-303}), which   generalises  the results of \cite{AS:1,GHM}. We also prove that every star-configuration is aCM (see Theorem~\ref{T:20140307-303}), which  generalises  the interesting result of \cite{GHM}.

As an application, we discuss the weak Lefschetz-property in Section 4.  
This fundamental property  has  been studied by many authors (see \cite{AS, AS:1,BZ, GHMS, GHS:2, HMNW, H:1,  MR:1, S:2}). In \cite{S:2}, the author proved that if $\X$ is the union of two star-configurations in $\P^2$ defined by linear forms and quadratic forms and $\sigma(\X)\ne \sigma(\Y)$, 
then $R/(I_\X+I_\Y)$ has the weak Lefschetz property. In this section we also find another Artinian ring $R/(I_\X+I_\Y)$ having the weak Lefschetz property without the condition $\sigma(\X)\ne\sigma(\Y)$ (see Proposition~\ref{P:20140322-411}). We also prove that if $\X$ and $\Y$ are star-configurations in $\P^2$ of type $(2,s)$ and $(2,s+1)$ defined by forms $F_1,\dots,F_s$, and $G_1,\dots,G_s,L$ in $R=\k[x_0,x_1,x_2]$, respectively, with
$\deg(F_i)=\deg(G_i)\le 2$ for $i=1,\dots,s$, and $L$ is a general linear form in $R$, then $R/(I_\X+I_\Y)$ has the weak Lefschetz property with a Lefschetz element $L$ (see Theorem~\ref{T:20121102-208}), which    generalises  the result of \cite{S:2}.

\section{The Minimal Generators of The Ideal of A Star-configurations in $\P^n$}

\begin{defn-rem} Let $R=\k[x_0,x_1,\dots,x_n]$ be a polynomial ring over a field $\k$. For positive integers $r$ and $s$ with $1 \le r \le \min\{n,s\}$, suppose $F_1,\dots,F_s$ are general forms in $R$ of degrees $d_1,\dots,d_s$,  respectively. We call the variety $\X$ defined by the ideal
$$
\bigcap_{1\le i_1<\cdots<i_{r}\le s} (F_{i_1},\dots,F_{i_r}) 
$$
{\em a star-configuration in $\P^n$} of type $(r,s)$. In particular, if $F_1,\dots,F_s$ are general linear forms in $R$, then we call $\X$ {\em a linear star-configuration in $\P^n$ of type $(r,s)$}. 

 Notice that each $n$-forms $F_{i_1},\dots,F_{i_n}$ define  $d_{i_1}\cdots d_{i_n}$ points in $\P^n$ for each $1\le i_1<\cdots<i_n \le s$. Thus the ideal
$$
\bigcap_{1\le i_1<\cdots<i_{n}\le s} (F_{i_1},\dots,F_{i_n}) 
$$
defines a finite set $\X$ of points in $\P^n$ with 
$$
\deg (\X) = \sum_{1\le i_1 < i_2 < \cdots < i_{n} \le s} d_{i_1} d_{i_2} \cdots d_{i_n}.
$$
 
\end{defn-rem}


\begin{lem}\label{L:20131014-202}  Let $F_1,\dots,F_s,G_1,\dots,G_t$ be general forms in $R=\k[x_0,x_1,\dots,x_n]$ with $s\ge 1$ and $n\ge 2$. 
Then 
$$
\begin{array}{llllllllllllllll}
\ds\bigcap_{1\le i \le s} (F_{i},G_1,\dots,G_t)= \left({\prod_{\ell=1}^{s} F_\ell},G_1,\dots,G_t\right)
\end{array} 
$$
for $0\le t \le n-1$.
\end{lem} 

\begin{proof}
We shall prove this by induction on $s$. If $s=1$, then it is clear. So we assume that $s>1$. Then, by induction on $s$, 
$$
\begin{array}{llllllllllllllllllllllllllllllll} 
&    & { \bigcap}_{1\le i \le s} (F_{i},G_1,\dots,G_t) \\[.5ex] 
& = & \big[{ \bigcap}_{1\le i\le s-1} (F_{i},G_1,\dots,G_t)\big]  \bigcap 
          (F_s,G_1,\dots,G_t)   \\[.5ex] 
& = &  \big({\prod_{\ell=1}^{s-1} F_\ell},G_1,\dots,G_t\big) \bigcap  (F_s,G_1,\dots,G_t).  
\end{array} 
$$
First, it is obvious that 
$$
\begin{matrix}
\big({\prod_{\ell=1}^{s} F_\ell},G_1,\dots,G_t\big) \subseteq 
          \big({\prod_{\ell=1}^{s-1} F_\ell},G_1,\dots,G_t\big) \bigcap (F_s,G_1,\dots,G_t).
\end{matrix} 
$$
Conversely, assume that $F\in \big({\prod_{\ell=1}^{s-1} F_\ell},G_1,\dots,G_t\big) \bigcap (F_s,G_1,\dots,G_t)$. Then, for some $H_i,K_i,L_i,M_i\in R$, 
$$
\begin{array}{llllllllllllllllllllllllllllll}
F =  \big({\prod_{\ell=1}^{s-1} F_\ell} \big)H_1+K_1G_1+\cdots+K_tG_t=F_{s}M_{s}+L_1G_1+\cdots+L_tG_t.
\end{array}
$$
Moreover, since $n\ge t+1$, the forms $\prod_{\ell=1}^{s-1} F_\ell$, $F_{s},G_1,\dots,G_t$  are general forms in $R$.  We thus get that
$H_1\in (F_s,G_1,\dots,G_t)$, and $F\in \big({\prod_{\ell=1}^{s} F_\ell},G_1,\dots,G_t\big)$, as we wished. 
\end{proof} 

\begin{thm} \label{T:20140302-103} 
Let $F_1,\dots,F_s,G_1,\dots,G_t$ be general forms in $R=\k[x_0,x_1,\dots,x_n]$ {for $s \ge 2$, $t\ge 0$, and $n\ge 2$}. Then 
$$
\bigcap_{1\le j_1<\cdots<j_{r}\le s} (F_{j_1},\dots,F_{j_{r}},G_1,\dots,G_t)= \sum_{1\le i_1<\cdots<i_{r-1} \le s}  \left(\frac{\prod_{\ell=1}^{s} F_\ell}{F_{i_1}\cdots F_{i_{r-1}}},G_1,\dots,G_t\right)
$$
for $1 \le r \le n-t$.  
\end{thm}

\begin{proof} We shall prove this by double induction on $r$ and $s$. If $r=1$, then by Lemma~\ref{L:20131014-202}, it holds for every $s\ge 1$. 
Now assume $r>1$. If $s=2$, then it is immediate. Let $s>2$. By double induction on $s$ and $r$, 
$$
\begin{array}{llllllllllllllllllll}
    & \ds \bigcap_{1\le j_1<\cdots<j_{r}\le s} (F_{j_1},\dots,F_{j_{r}},G_1,\dots,G_t) \\[3ex] 
= & \ds\Big[\bigcap_{1\le j_1<\cdots<j_{r}\le s-1} (F_{j_1},\dots,F_{j_{r}},G_1,\dots,G_t) \Big]\bigcap 
          \Big[\bigcap_{1\le j_1<\cdots<j_{r-1}\le s-1} (F_{j_1},\dots,F_{j_{r-1}},F_{s},G_1,\dots,G_t) \Big] \\
= &\ds\Big[\ds\sum_{1\le i_1<\cdots<i_{r-1} \le s-1}  \Big(\frac{\prod_{\ell=1}^{s-1} F_\ell}{F_{i_1}\cdots F_{i_{r-1}}},G_1,\dots,G_t\Big)\Big]
\bigcap \Big[\ds\sum_{1\le i_1<\cdots<i_{r-2} \le s-1}  \Big(\frac{\prod_{\ell=1}^{s-1} F_\ell}{F_{i_1}\cdots F_{i_{r-2}}},F_s,G_1,\dots,G_t\Big)\Big].
\end{array} 
$$

First, note that
$$
\begin{array}{llll}  
& \ds\sum_{1\le i_1<\cdots<i_{r-1} \le s}  \Big(\frac{\prod_{\ell=1}^{s} F_\ell}{F_{i_1}\cdots F_{i_{r-1}}},G_1,\dots,G_t\Big)  \\[1ex] 
\subseteq &
\ds\Big[\ds\sum_{1\le i_1<\cdots<i_{r-1} \le s-1}  \Big(\frac{\prod_{\ell=1}^{s-1} F_\ell}{F_{i_1}\cdots F_{i_{r-1}}},G_1,\dots,G_t\Big)\Big]
\bigcap \Big[\ds\sum_{1\le i_1<\cdots<i_{r-2} \le s-1}  \Big(\frac{\prod_{\ell=1}^{s-1} F_\ell}{F_{i_1}\cdots F_{i_{r-2}}},F_s,G_1,\dots,G_t\Big)\Big].
\end{array}
$$
Conversely, assume that 
$$
\begin{matrix} 
F\in \ds\Big[\ds\sum_{1\le i_1<\cdots<i_{r-1} \le s-1}  \Big(\frac{\prod_{\ell=1}^{s-1} F_\ell}{F_{i_1}\cdots F_{i_{r-1}}},G_1,\dots,G_t\Big)\Big]
\bigcap \Big[\ds\sum_{1\le i_1<\cdots<i_{r-2} \le s-1}  \Big(\frac{\prod_{\ell=1}^{s-1} F_\ell}{F_{i_1}\cdots F_{i_{r-2}}},F_s,G_1,\dots,G_t\Big)\Big].
\end{matrix} 
$$
Then for some $K_{i_1\cdots i_{r-1}}, M_{j_1\cdots j_{r-2}},H_i,L_i\in R$,

\begin{align}
F 
=  & \ds\sum_{1\le i_1<\cdots<i_{r-1} \le s-1} \frac{\prod_{\ell=1}^{s-1} F_\ell}{F_{i_1}\cdots F_{i_{r-1}}}\cdot K_{i_1\cdots i_{r-1}}+\sum^t_{i=1} H_iG_i
\label{EQ:20140303-101}  \\ 
=  &\ds\sum_{1\le j_1<\cdots<j_{r-2} \le s-1} \frac{\prod_{\ell=1}^{s-1} F_\ell}{F_{j_1}\cdots F_{j_{r-2}}}\cdot M_{j_1\cdots j_{r-2}}+F_sM_s
+\sum^t_{i=1} L_iG_i. \label{EQ:20140303-102} 
\end{align} 

\medskip
 
 Now we first show that, using two representation of $F$ given in \eqref{EQ:20140303-101} and \eqref{EQ:20140303-102}, $K_{i_1\ldots i_{r-1}}$ belongs to the ideal $(F_{i_1}, \dots, F_{i_{r-1}}, F_s, G_1, \dots, G_t)$. Indeed if $(i_1',i_2',\dots,i_{r-1}')\ne (i_1,i_2,\dots,i_{r-1})$, then  
$$
\frac{\prod_{\ell=1}^{s-1} F_\ell}{F_{i_1'}F_{i_2'} \cdots F_{i_{r-1}'}} \in (F_{i_1},F_{i_2},\dots,F_{i_{r-1}}).
$$
So,
\begin{equation} \label{EQ:20140307-110} 
\begin{array}{llllllll}
&    \ds\sum_{\substack{
1\le i_1'<i_{2}'<\cdots<i_{r-1}'\le s-1\\ (i_1',i_2',\dots,i_{r-1}')\ne (i_1,i_2,\dots,i_{r-1})}
} \frac{\prod_{\ell=1}^{s-1} F_\ell}{F_{i_1'}F_{i_2'} \cdots F_{i_{r-1}'}}\cdot K_{i_1'i_2'\cdots i_{r-1}'} \in (F_{i_1},F_{i_2},\dots,F_{i_{r-1}}).
\end{array} 
\end{equation}
Moreover, we have that, for every $1\le j_1<\cdots <j_{r-2}\le s-1$,
$$
\begin{array}{llllllll}
\{i_1,\dots,i_{r-1}\}\backslash \{j_1,\dots,j_{r-2}\}\ne \varnothing, \text{ i.e., } \\
\ds\sum_{1\le j_1<\cdots<j_{r-2} \le s-1} \frac{\prod_{\ell=1}^{s-1} F_\ell}{F_{j_1}\cdots F_{j_{r-2}}}\cdot M_{j_1\cdots j_{r-2}}
\in (F_{i_1},F_{i_2},\dots,F_{i_{r-1}}).
   \end{array} 
$$
It follows from the second representation of $F$ in the equation~\eqref{EQ:20140303-102} that 
\begin{equation} \label{EQ:20131022-223} 
\begin{array}{lllllllllll}
F
 & =  & \ds\sum_{1\le j_1<\cdots<j_{r-2} \le s-1} \frac{\prod_{\ell=1}^{s-1} F_\ell}{F_{j_1}\cdots F_{j_{r-2}}}\cdot M_{j_1\cdots j_{r-2}}+F_sM_s +\sum^t_{i=1} L_iG_i.\\
  & \in &  (F_{i_1},F_{i_2},\dots,F_{i_{r-1}},F_s,G_1,\dots,G_t).
\end{array} 
\end{equation}    
Recall that $F_{i_1},F_{i_2},\dots,F_{i_{r-1}},F_s,G_1,\dots,G_t$ are $(r+t)$-general forms in $R$ with $r+t\le n$. Hence it follows from the equations~\eqref{EQ:20140303-101}, \eqref{EQ:20140307-110},  and \eqref{EQ:20131022-223} that
$$
K_{i_1i_2\cdots i_{r-2}i_{r-1}}=F_{i_1}N_{i_1}+F_{i_2}N_{i_2}+\cdots+F_{i_{r-1}}N_{i_{r-1}}+F_sN_s+\sum^t_{i=1} N_i'G_i
$$
for some $N_{i_j},N_i' \in R$. Thus we have
\begin{equation} \label{EQ:20131022-224} 
\begin{array}{llllllllllll}
   & \frac{\prod_{\ell=1}^{s-1} F_\ell}{F_{i_1}F_{i_2} F_{i_3}\cdots F_{i_{r-2}}F_{i_{r-1}}}\cdot K_{i_1i_2i_3\cdots i_{r-2}i_{r-1}} \\[1.2ex]
= & \frac{\prod_{\ell=1}^{s-1} F_\ell}{F_{i_1}F_{i_2} F_{i_3}\cdots F_{i_{r-2}}F_{i_{r-1}}}\cdot(F_{i_1}N_{i_1}+F_{i_2}N_{i_2}+F_{i_3}N_{i_3}+\cdots+F_{i_{r-1}}N_{i_{r-1}}+F_sN_s+\sum^t_{i=1} N_i'G_i)\\[1.2ex] 
=& \frac{\prod_{\ell=1}^{s} F_\ell}{F_{i_2} F_{i_3}\cdots F_{i_{r-2}}F_{i_{r-1}}F_s}\cdot N_{i_1} +
\frac{\prod_{\ell=1}^{s} F_\ell}{F_{i_1}  F_{i_3}\cdots F_{i_{r-2}}F_{i_{r-1}}F_s}\cdot  N_{i_2} +
\frac{\prod_{\ell=1}^{s} F_\ell}{F_{i_1}F_{i_2} F_{i_4} \cdots F_{i_{r-2}}F_{i_{r-1}}F_s}\cdot  N_{i_3} + \cdots+\\[1.2ex] 
&  \frac{\prod_{\ell=1}^{s} F_\ell}{F_{i_1}F_{i_2} F_{i_3}\cdots F_{i_{r-2}} F_s}\cdot  N_{i_{r-1}} +
 \frac{\prod_{\ell=1}^{s} F_\ell}{F_{i_1}F_{i_2} F_{i_3}\cdots F_{i_{r-2}}F_{i_{r-1}}} \cdot N_s +\frac{\prod_{\ell=1}^{s-1} F_\ell}{F_{i_1}F_{i_2} F_{i_3}\cdots F_{i_{r-2}}F_{i_{r-1}}} \cdot \sum^t_{i=1} N_i'G_i\\[1.2ex] 
 \in  & {\ds\sum}_{1\le k_1<\cdots<k_{r-1} \le s}  \Big(\frac{\prod_{\ell=1}^{s} F_\ell}{F_{k_1}\cdots F_{k_{r-1}}},G_1,\dots,G_t\Big).
\end{array}
\end{equation} 
This holds for arbitrary chosen $(i_1,\ldots,i_{r-1})$ with $1\le i_1 < \cdots < i_{r-1} \le s$. This implies that 
$$
F 
 \in  \ds\sum_{1\le i_1<\cdots<i_{r-1} \le s}  \bigg(\frac{\prod_{\ell=1}^{s} F_\ell}{F_{i_1}\cdots F_{i_{r-1}}},G_1,\dots,G_t\bigg),
$$
as we wished. This completes the proof. 
\end{proof}

The following corollary is immediate from Theorem~\ref{T:20140302-103} .

\begin{cor}\label{C:20131022-205}  Let $\X$ be a star-configuration in $\P^n$ of type $(r,s)$ defined by general forms $F_1,\dots,F_s$ in $R=\k[x_0,x_1,\dots,x_n]$ with $1\le r\le \min\{s,n\}$. Then
$$
I_\X=\bigcap_{1\le i_1<\cdots <i_r\le s}(F_{i_1},\dots,F_{i_r})=\sum_{1\le j_1<\cdots<j_{r-1}\le s} \left(\frac{\prod^s_{\ell=1}F_\ell}{F_{j_1}\cdots F_{j_{r-1}}}\right). 
$$
\end{cor} 


 \begin{cor} \label{C:20121023-301} Let $\X$ be a linear star-configuration in $\P^n$ of type $(n,s)$ defined by $s$ general linear forms in $R=\k[x_0,\dots,x_n]$ with $s\ge n\ge 2$. Then $\X$ has  generic Hilbert function
$$
\begin{array}{lllllllllllllllllllllll}
\H_\X & : & 1 & \ds\binom{1+n}{n} & \cdots &\ds \binom{(s-n)+n}{n} & \ds\binom{(s-n)+n}{n} & \ra,
\end{array} 
$$
i.e.,
$$
\H_\X(i)=\min\bigg\{\deg(\X), \binom{i+n}{n}\bigg\}
$$
for every $i\ge 0$. 
\end{cor}
\begin{proof}    By Corollary~\ref{C:20131022-205},  $I_\X$ has exactly $\binom{s}{n-1}$-generators in degree $s-(n-1)$. Thus $\H_\X(i) = \binom{i+n}{n}$ for $i \le s-n$. Moreover, since
$$ 
\H_\X(s-(n-1)) = \binom{s+1}{n} - \binom{s}{n-1} = \binom{(s-n)+n}{n} = \deg(\X),
$$
we see that $\X$ has generic Hilbert function, as we wished.  
\end{proof} 

The following example shows that the Hilbert function of a  star-configuration in $\P^n$ of type $(n,s)$ is not generic in general. 
 
\begin{exmp} Consider a star-configuration in $\P^3$  of type $(3,3)$ defined by $3$ general quadratic forms. In this case, $I_\X$ has exactly $3$ generators in degree $2$, i.e., $\H_\X (2) = \binom{5}{3} - 3 = 7$. But $\deg(\X) = 8$, and so the Hilbert function of $\X$ is
$$
\begin{array}{llllllllllllllllllllll}
\H_\X & : & 1 & 4 & 7 & 8 & \ra, 
\end{array} 
$$
which is not generic. 
\end{exmp}

\section{The Minimal Free Resolution of A  Star-configuration in $\P^n$}

We first recall from \cite{KMMNP} the following result. 

\begin{pro} \label{P:201450307-201} Let $I_C$ be a saturated ideal defining a codimension $c$ subscheme $C\subseteq \P^n$. Let $I_S\subset I_C$ be an ideal which defines  an aCM subscheme $S$ of codimension $c-1$. Let $F$ be a form of degree $d$ which is not a zero divisor on $R/I_S$. Consider the ideal $I'=F\cdot I_C+I_S$ and let $C'$ be the subscheme it defines. Then $I'$ is saturated, hence equal to $I_{C'}$, and there is an exact sequence
$$
0\to I_S(-d) \to I_C (-d) \oplus I_S \to I_{C'} \to 0.
$$
In particular, since $S$ is an aCM subscheme of codimension one less than $C$, we see that $C'$ is an aCM subscheme if and only if $C$ is. Also 
$$
\deg C'=\deg C+(\deg F) \cdot (\deg S).
$$
Furthermore, as sets on $S$, we have $C'=C' \cup H_F$, where $H_F$ is the hyper surface section cut out on $S$ by $F$. The Hilbert function $\H_{C'}$ of $R/I_{C'}$ is 
$$
\H_{C'}(t)=\H_S(t)-\H_S(t-d)+\H_C(t-d). 
$$
\end{pro} 

\begin{rem}
The construction in Proposition~\ref{P:201450307-201} is often referred to as {\em Basic Double $G$-Linkage}. 
\end{rem} 


\begin{thm}[{\cite[Theorem 2.1]{AS:1}}]\label{T:20110219-204}
Let $\X$ be a star-configuration in $\P^n$ of type $(2,s)$ defined by general forms in $R=\k[x_0,x_1,\dots,x_n]$ of degrees $d_1,\dots, d_s$, and let $d=d_1+d_2+\cdots+d_s$. Then the  minimal free resolution of $R/I_\X$ is
$$
\begin{matrix}
0 & \ra & R^{s-1}(-d)&  {\ra} & \bigoplus_{i=1}^s R(-(d-d_i)) &  {\ra} & R & \ra & R/I_\X & \ra & 0.
\end{matrix} 
$$
\end{thm}

\begin{rem} Let $L_1,\dots,L_s$ be general linear forms in $R=\k[x_0,x_1,\dots,x_n]$, and let $\Y$ be a star-configuration in $\P^n$ of type $(2,s)$ defined by forms $L_1^{d_1},\dots,L_s^{d_s}$  of degrees $d_1,\dots,d_s$. By Theorem 3.1 in \cite{AS:2}, $R/I_\Y$ has the same minimal free resolution as $R/I_\X$, where $\X$ is a star-configuration in $\P^n$ of type $(2,s)$ defined by general forms $F_1,\dots,F_s$ in $R$ of degrees $d_1,\dots,d_s$, respectively. Using this result, we obtain the following corollary for a specific case.
\end{rem}

\begin{cor} \label{C:20130107-203} Let $\X$ be a star-configuration in $\P^n$ of type $(2,s)$ defined by general forms $F_1,\dots,F_s$ in $R$ of degree $d$, and let $\Y$ be a star-configuration in $\P^n$ of type $(2,s)$ defined by forms $L_1^d,\dots,L_s^d$, where $L_i$'s are general linear forms in $R$. Then $R/I_\X$ and $R/I_\Y$ have the same minimal free resolution.
\end{cor}

\begin{thm} \label{T:20140307-303} Let $\X^{(r,s)}$ be a star-configuration in $\P^n$ of type $(r,s)$ defined by general forms $F_1,\dots,F_s$ in $R=\k[x_0,x_1,\dots,x_n]$ of degrees $d_1,d_2,\dots, d_s,$ where  $2\le r \le \min\{s,n\}$, and let $d=d_1+\cdots+d_s$. Then the minimal free resolution of $I_{\X^{(r,s)}}$ is 
\begin{equation} \label{EQ:20140328-301}
\begin{array}{ccccccccccccccc} 
0 \to \FF_{r}^{(r,s)} \to \FF_{r-1}^{(r,s)}  \to \cdots \to \FF_1^{(r,s)}  \to I_{\X^{(r,s)}} \to 0
\end{array} 
\end{equation} 
where
$$
\begin{array}{rcllllllllllll}
\FF_r^{(r,s)}  & = & R^{\alpha_{r}^{(r,s)}} (-d), \\
\FF_{r-1}^{(r,s)}  & = & \ds \bigoplus_{1\le i_1\le s}  R^{\alpha_{r-1}^{(r,s)}}(-(d-d_{i_1})),\\
& \vdots &  \\
\FF_{\ell}^{(r,s)}  & = &\ds  \bigoplus_{1\le i_1<\cdots < i_{r-\ell}\le s} R^{\alpha_{\ell}^{(r,s)}} (-(d-(d_{i_1}+\cdots+d_{i_{r-\ell}}))), \\
& \vdots &  \\
\FF_{2}^{(r,s)}  & = & \ds  \bigoplus_{1\le i_1<\cdots < i_{r-2}\le s} R^{\alpha_{2}^{(r,s)}} (-(d-(d_{i_1}+\cdots+d_{i_{r-2}}))),  \quad \text{and} \\
\FF_{1}^{(r,s)}  & = &  \ds  \bigoplus_{1\le i_1<\cdots < i_{r-1}\le s} R (-(d-(d_{i_1}+\cdots+d_{i_{r-1}}))), 
\end{array} 
$$
 with
$$
{\alpha_{\ell}^{(r,s)}}=\binom{s-r+\ell-1}{\ell-1} \quad \text{and} \quad {\rm rank}  \FF_\ell^{(r,s)}=\binom{s-r+\ell-1}{\ell-1}\cdot \binom{s}{r-\ell}
$$
for $1\le \ell \le r$. In particular, the last free module $\FF_r^{(r,s)}$ has only one shift $d$, i.e., a star-configuration $\X^{(r,s)}$ in $\P^n$ is level. Furthermore, any star-configuration $\X^{(r,s)}$ in $\P^n$ is aCM. 

\end{thm}   

\begin{proof} We shall prove this by double induction on $r$ and $s$. If $r=2$, then, by Theorem~\ref{T:20110219-204},  the result holds.
Now assume $r>2$. 
Let $\X^{(r,s-1)}$ be a star-configuration in $\P^n$ of type $(r,s-1)$  and $\X^{(r-1,s-1)}$ a star-configuration in $\P^n$ of type $(r-1,s-1)$ defined by general forms of degrees $d_1,\dots,d_{s-1}$, respectively.

For the initial case of $s$, i.e., $r=s$, the minimal free resolution of $I_{\X^{(r,s)}}$ is obtained from the Koszul-complex generated by a regular sequence of general forms of degrees $d_1,\dots,d_s$, and so
$$
{\alpha_{\ell}^{(r,r)}}=1=\binom{\ell-1}{\ell-1} \quad \text{and} \quad  {\rm rank} \, \FF_\ell^{(r,r)}=\binom{r}{r-\ell}=\binom{\ell-1}{\ell-1}\cdot \binom{r}{r-\ell},
$$
as we wished.

Now suppose $r<s$. By Proposition~\ref{P:201450307-201}, we obtain  the exact sequence
\begin{equation} \label{EQ:20140320-301} 
0 \to I_{\X^{(r-1,s-1)}}(-d_s) \to  I_{\X^{(r,s-1)}}(-d_s) \oplus I_{\X^{(r-1,s-1)}} \to I \to 0,
\end{equation} 
where 
$$
I=F_s\cdot I_{\X^{(r,s-1)}}+I_{\X^{(r-1,s-1)}}.
$$

Notice that, by Corollary~\ref{C:20131022-205}, $I$ is the ideal of a star-configuration $\X^{(r,s)}$ in $\P^n$ of type $(r,s)$ defined by general forms $F_1,\dots,F_s$ in $R$ of degrees $d_1,\dots,d_s$, i.e., $I=I_{\X^{(r,s)}}$. Let $d'=d_1+\cdots+d_{s-1}$. By double induction on $r$ and $s$, we assume that 
$$
\begin{array}{cccccccllllllllll}
0 \to \FF_{r-1}^{(r-1,s-1)}  \to \cdots \to \FF_1^{(r-1,s-1)}  \to I_{\X^{(r-1,s-1)}} \to 0, \quad \text{and} \\
0 \to \FF_{r}^{(r,s-1)}  \to \cdots \to \FF_1^{(r,s-1)}  \to I_{\X^{(r,s-1)}} \to 0
\end{array} 
$$
are free resolutions of $\X^{(r-1,s-1)}$ and $\X^{(r,s-1)}$, respectively, such that
$$
\begin{array}{ccllllllllllll}
\FF_{r-1}^{(r-1,s-1)}  & = & \ds R^{\alpha_{r-1}^{(r-1,s-1)}} (-d'), \\
\FF_{r-2}^{(r-1,s-1)}  & = & \ds \bigoplus_{1\le i_1\le s-1}  R^{\alpha_{r-2}^{(r-1,s-1)}}(-(d'-d_{i_1})),\\
& \vdots &  \\
\FF_{\ell}^{(r-1,s-1)}  & = &\ds  \bigoplus_{1\le i_1<\cdots < i_{(r-1)-\ell}\le s-1} R^{\alpha_\ell^{(r-1,s-1)}} (-(d'-(d_{i_1}+\cdots+d_{i_{(r-1)-\ell}}))), \\
& \vdots &  \\
\FF_{2}^{(r-1,s-1)}  & = & \ds  \bigoplus_{1\le i_1<\cdots < i_{r-3}\le s-1} R^{\alpha_2^{(r-1,s-1)}} (-(d'-(d_{i_1}+\cdots+d_{i_{r-3}}))),  \\
\FF_{1}^{(r-1,s-1)}  & = &  \ds  \bigoplus_{1\le i_1<\cdots < i_{r-2}\le s-1} R (-(d'-(d_{i_1}+\cdots+d_{i_{r-2}}))), 
\end{array} 
$$
and
$$
\begin{array}{rcllllllllllll}
\FF_r^{(r,s-1)}  & = & \ds R^{\alpha_{r}^{(r,s-1)}}(-d'), \\
\FF_{r-1} ^{(r,s-1)} & = & \ds \bigoplus_{1\le i_1\le s-1}  R^{{\alpha_{r-1}^{(r,s-1)}}}(-(d'-d_{i_1})),\\
& \vdots &  \\
\FF_{\ell}^{(r,s-1)}  & = &\ds  \bigoplus_{1\le i_1<\cdots < i_{r-\ell}\le s-1} R^{\alpha_{\ell}^{(r,s-1)}} (-(d'-(d_{i_1}+\cdots+d_{i_{r-\ell}}))), \\
& \vdots &  \\
\FF_{2}^{(r,s-1)}  & = & \ds  \bigoplus_{1\le i_1<\cdots < i_{r-2}\le s-1} R^{\alpha_{2}^{(r,s-1)}} (-(d'-(d_{i_1}+\cdots+d_{i_{r-2}}))),  \\
\FF_{1}^{(r,s-1)}  & = &  \ds  \bigoplus_{1\le i_1<\cdots < i_{r-1}\le s-1} R (-(d'-(d_{i_1}+\cdots+d_{i_{r-1}}))).
\end{array} 
$$
By the mapping cone construction with equation~\eqref{EQ:20140320-301}, we obtain the following diagram:
$$
\begin{array}{ccccccccccccccc}
&&   && 0 &&   \\
&&   && \downarrow &&   \\ 
&& 0 && \FF_r^{(r,s-1)} (-d_s) &&    \\
&& \downarrow && \downarrow &&  \\ 
&& \FF_{r-1}^{(r-1,s-1)} (-d_s) && \FF_{r-1}^{(r,s-1)} (-d_s)\oplus \FF_{r-1}^{(r-1,s-1)}  &&   \\
&& \downarrow && \downarrow &&  \\ 
&& \vdots && \vdots &&   \\ 
&& \FF_{1}^{(r-1,s-1)} (-d_s) && \FF_{1}^{(r,s-1)} (-d_s)\oplus \FF_{1}^{(r-1,s-1)}  &&   \\
&& \downarrow && \downarrow &&  \\ 
0 & \to & I_{\X^{(r-1,s-1)}}(-d_s) & \to  & I_{\X^{(r,s-1)}}(-d_s) \oplus I_{\X^{(r-1,s-1)}} & \to & I & \to & 0. \\
&& \downarrow && \downarrow && \\ 
&&0 && 0 &&  \\ 
\end{array} 
$$
Hence we obtain a free resolution of $I=I_{\X^{(r,s)}}$ as
\begin{equation} \label{EQ:20140320-302} 
\begin{array}{lllllllllllllllllllllll}
\begin{matrix} 
0 \to 
\left[\begin{matrix}  \FF_r^{(r,s-1)} (-d_s) \\ \oplus \\ \FF_{r-1}^{(r-1,s-1)} (-d_s) \end{matrix} \right] \to  
\left[\begin{matrix}  \FF_{r-1}^{(r,s-1)} (-d_s)\oplus \FF_{r-1}^{(r-1,s-1)}  \\ \oplus \\ \FF_{r-2}^{(r-1,s-1)} (-d_s) \end{matrix}\right]   \to  
\cdots \to 
\left[\begin{matrix}  \FF_{2}^{(r,s-1)} (-d_s)\oplus \FF_{2}^{(r-1,s-1)}  \\ \oplus \\ \FF_{1}^{(r-1,s-1)} (-d_s) \end{matrix}  \right]
\end{matrix}   
\\[4ex] 
\begin{matrix} 
\phantom{0} \to  \FF_{1} ^{(r,s-1)} (-d_s)\oplus \FF_{1}^{(r-1,s-1)}    \to 
I  \to  0. 
\end{matrix} 
\end{array} 
\end{equation}

Now consider a free module 
\begin{equation} \label{EQ:20140319-303} 
\begin{array}{ccccccccccc}
\FF_\ell^{(r,s)}
& = & \FF_\ell^{(r,s-1)}(-d_s) \oplus \FF_{\ell}^{(r-1,s-1)} \oplus \FF_{\ell-1}^{(r-1,s-1)}(-d_s), \\[1ex] 
& = & \left[
\begin{matrix} 
\ds  \bigoplus_{1\le i_1<\cdots < i_{r-\ell}\le s-1} R^{\alpha_\ell^{(r,s-1)}} (-(d'+d_s-(d_{i_1}+\cdots+d_{i_{r-\ell}})))\\
\oplus \\
\ds  \bigoplus_{1\le i_1<\cdots < i_{(r-1)-\ell}\le s-1} R^{\alpha_\ell^{(r-1,s-1)}} (-(d'-(d_{i_1}+\cdots+d_{i_{(r-1)-\ell}}))) \\
\oplus \\
\ds  \bigoplus_{1\le i_1<\cdots < i_{(r-1)-(\ell-1)}\le s-1} R^{\alpha_{\ell-1}^{(r-1,s-1)}} (-(d'+d_s-(d_{i_1}+\cdots+d_{i_{(r-1)-(\ell-1)}}))).
\end{matrix} 
\right] 
\end{array}
\end{equation} 
for $1\le \ell\le s$.
 Since $d=d'+d_s$, we can rewrite equation~\eqref{EQ:20140319-303} as 
\begin{equation} \label{EQ:20140320-303} 
\begin{array}{lllllllllllllll}
\FF_\ell^{(r,s)}
& = & \left[
\begin{matrix} 
\ds  \bigoplus_{1\le i_1<\cdots < i_{r-\ell}\le s-1} R^{\alpha_\ell^{(r,s-1)}} (-(d-(d_{i_1}+\cdots+d_{i_{r-\ell}})))\\
\oplus \\
\ds  \bigoplus_{1\le i_1<\cdots < i_{(r-1)-\ell}\le s-1} R^{\alpha_\ell^{(r-1,s-1)}} (-(d-(d_{i_1}+\cdots+d_{i_{(r-1)-\ell}}+d_s))) \\
\oplus \\
\ds  \bigoplus_{1\le i_1<\cdots < i_{r-\ell}\le s-1} R^{\alpha_{\ell-1}^{(r-1,s-1)}} (-(d-(d_{i_1}+\cdots+d_{i_{r-\ell}})))
\end{matrix} \right] \\[10.2ex] 
& = & \left[
\begin{matrix} 
\ds  \bigoplus_{1\le i_1<\cdots < i_{r-\ell}\le s-1} R^{\alpha_{\ell-1}^{(r-1,s-1)}+\alpha_\ell^{(r,s-1)}} (-(d-(d_{i_1}+\cdots+d_{i_{r-\ell}})))\\
\oplus \\
\ds  \bigoplus_{1\le i_1<\cdots < i_{(r-1)-\ell}\le s-1} R^{\alpha_\ell^{(r-1,s-1)}} (-(d-(d_{i_1}+\cdots+d_{i_{(r-1)-\ell}}+d_s))) 
\end{matrix} \right] .
\end{array}
\end{equation} 

Now we shall prove that $\alpha_\ell^{(r,s)}:=\alpha_{\ell-1}^{(r-1,s-1)}+\alpha_\ell^{(r,s-1)}={\alpha_\ell^{(r-1,s-1)}}$, i.e., $\FF_\ell^{(r,s)}$ is of the form
$$
\FF_\ell ^{(r,s)}
  =  \ds  \bigoplus_{1\le i_1<\cdots < i_{r-\ell}\le s} R^{\alpha_\ell^{(r,s)}} (-(d-(d_{i_1}+\cdots+d_{i_{r-\ell}}))) .
$$
If $\ell=1$, then by Corollary~\ref{C:20131022-205}
$$
\alpha_{1}^{(r,s)}=1=\binom{s-r+1-1}{1-1}, \quad \text{i.e.,} \quad {\rm rank}\FF_1=\binom{s}{r-1}=\binom{s}{r-1}\cdot \binom{s-r+1-1}{1-1}.
$$
Now suppose $1<\ell <r$. Then by double induction on $r$ and $s$, we have that 
$$
\begin{array}{lllllllllllllllllllllllll}
\alpha_{\ell-1}^{(r-1,s-1)}+\alpha_\ell^{(r,s-1)}
& = & \ds \binom{(s-1)-(r-1)+(\ell-1)-1}{(\ell-1)-1}+\binom{(s-1)-r+\ell-1}{\ell-1} \\[1.8ex] 
& = & \ds \binom{s-r+\ell-2}{\ell-2}+\binom{s-r+\ell-2}{\ell-1} \\[1.8ex]
& = & \ds \binom{s-r+\ell-1}{\ell-1}\\[1.8ex]
& = & \ds \binom{(s-1)-(r-1)+\ell-1}{\ell-1}\\[1.8ex]
& = & \alpha_{\ell}^{(r-1,s-1)},
\end{array} 
$$
and hence
$$
{\rm rank}\FF_\ell=\ds \binom{s-r+\ell-1}{\ell-1}\cdot \ds \binom{s}{r-\ell}.
$$
Moreover, if $\ell=r$, then it is from equation~\eqref{EQ:20140320-302} and double induction on $r$ and $s$ that
$$
\begin{array}{lllllllllllllllllllllll}
\alpha_{r}^{(r,s)}
& = & \alpha^{(r,s-1)}_r+\alpha_{r-1}^{(r-1,s-1)} \\[.5ex] 
& = & \ds \binom{s-2}{r-1}+\binom{s-2}{r-2} \\[1.8ex] 
& = & \ds \binom{s-1}{r-1}\\[1.8ex] 
& = & {\rm rank} \, \FF_r^{(r,s)},
\end{array} 
$$
as we wished. 

Now we shall prove that the free resolution in equation~\eqref{EQ:20140320-302} is  minimal. Since $\k$ is an infinite field, for any $1\le r \le \min\{s,n\}$, we can take an $(s-r+1)\times s$ matrix $A=[a_{i,j}]$ such that any $\gamma\times \gamma$ minor is not $0$ for some $a_{i,j}\in \k$ and for every $1\le \gamma\le s-r+1$. Define a  $(s-r+1)\times s$ matrix $M=[a_{i,j}F_j]$. Then, a $\gamma\times \gamma$-minor of $M$ is of the form
 $$
 \begin{array}{lllllllllllll}
 & \det
  \left[
  \begin{array}{rrrrrrrrrrrr}
 a_{i_11,j_1}F_{j_1} & \cdots &  a_{i_1,j_\gamma F_{j_\gamma}} \\
 & \dots & \\
  a_{i_\gamma,j_1}F_{j_1} & \cdots &  a_{i_\gamma,j_\gamma}F_{\gamma}
 \end{array} 
 \right] \\[3ex]  
 = & 
   \det\left[
  \begin{array}{rrrrrrrrrrrr}
 a_{i_11,j_1}  & \cdots &  a_{i_1,j_\gamma}   \\
 & \dots & \\
  a_{i_\gamma,j_1}  & \cdots &  a_{i_\gamma,j_\gamma} 
 \end{array}  \right]  
\cdot F_{j_1} \cdots F_{j_\gamma}. 
 \end{array} 
 $$
Since any $\gamma\times \gamma$ minor of the matrix  $A$ is not $0$, by Corollary~\ref{C:20131022-205} we get that the ideal generated by all maximal minors of the matrix $M$ is $I_\X$. Since $\depth \, I_{\X} = r$, by Corollary A2.12 in \cite{E:1} the graded Betti numbers of the homogeneous coordinate ring of $\X:=\X^{(r,s)}$ are those given by the Eagon-Northcott resolution of the $\gamma\times \gamma$ minors of $M$. In other words, the free resolution of $I_\X$ in equation~\eqref{EQ:20140320-302} is minimal. Moreover, since the last free module $\FF_r^{(r,s)}$ has only one shift $d$,  any star-configuration ${\X^{(r,s)}}$ in $\P^n$ is level.

For the last assertion, recall that, by Theorem~\ref{T:20110219-204}, any star-configuration in $\P^n$ of type $(2,s)$ (i.e., codimension $2$) is aCM. Suppose $r>2$. If $r=s$, then $\X^{(r,s)}$ is a complete intersection, i.e., $\X^{(r,s)}$ is aCM.
If $r<s$, then by double induction on $r$ and $s$, we assume that ${\X^{(r-1,s-1)}}$ and ${\X^{(r,s-1)}}$ are aCM, and so, by Proposition~\ref{P:201450307-201}, ${\X^{(r,s)}}$ is also aCM, which completes this theorem.
\end{proof} 

As a corollary, we obtain the result in \cite{GHM} with $d_1=\cdots=d_s=1$ in Theorem~\ref{T:20140307-303} and we omit the proof.

\begin{cor}[{\cite[Remark 2.11]{GHM}}] \label{C:20140322-305} Let $\X:=\X^{(r,s)}$ be a linear star-configuration in $\P^n$ of type $(r,s)$ with $1\le r\le \min\{s,n\}$. Then the minimal free resolution of $I_\X$ is given by
$$
\begin{array}{ccccccccccccccc} 
0 \to \FF_{r}^{(r,s)} \to \FF_{r-1}^{(r,s)}  \to \cdots \to \FF_1^{(r,s)}  \to I_{\X^{(r,s)}} \to 0
\end{array} 
$$
in Theorem~\ref{T:20140307-303} with $d_1=\cdots=d_s=1$. 
\end{cor} 

\section{The Weak-Lefschetz Property}
 
A graded Artinian $\k$-algebra $A=\bigoplus^s_{i=0} A_i\ (A_s \ne 0)$ has the {\em weak-Lefschetz property} if the homomorphism $(\times L):A_i \ra A_{i+1}$, induced by multiplication by a linear form $L$, has maximal rank for each $i$. In this case, we call $L$ a {\em Lefschetz element}.  

We first recall a question in \cite{AS:1}.

\begin{ques}[{\cite[Question 1.3]{AS:1}}]\label{Q:20110219-102} Let $\X:=\X^{(2,s)}$ and $\Y:=\X^{(2,t)}$ be star-configurations in $\P^2$ of type $(2,s)$ and $(2,t)$ defined by general forms of degree $d\ge1$ with $s\ge 3$. Does an Artinian ring $R/(I_{\X}+I_{\Y})$ have the weak-Lefschetz property? 
\end{ques}

We revise the above question to a more general question as follows.

\begin{ques} \label{Q:20131219-102} Let $\X:=\X^{(n,s)}$ and $\Y:=\X^{(n,t)}$ be star-configurations in $\P^n$ of type $(n,s)$ and $(n,t)$ defined by general forms of degree $d\ge1$. Does the Artinian ring $R/(I_{\X}+I_{\Y})$ have the weak-Lefschetz property? 
\end{ques}

We start with a proposition on the weak-Lefschetz property from \cite{GHMS} and provide an answer to Question~\ref{Q:20131219-102} for $d=1$. 
Let $\X$ be a finite set of points in $\P^n$ and define
$$
\sigma(\X)=\min\{\, i \mid \H_\X(i-1)=\H_\X(i)  \}.
$$

\begin{pro}[{\cite[Proposition 5.15]{GHMS}}] \label{P:202-20100706} Let $\X$ be a finite set of points in $\P^n$ and let $A$ be an Artinian quotient of the coordinate ring of $\X$. Assume that $\H_A(i)=\H_\X(i)$ for all $0 \le i \le \sigma(\X)-1$. Then $A$ has the weak-Lefschetz property. 
\end{pro}

By the same argument as in the proof of Theorem 4.2 in \cite{AS:1} , we obtain the following theorem immediately  and thus  omit the proof.  

\begin{thm}\label{T:204-20100706} Let $\X:=\X^{(n,s)}$ and $\Y:=\X^{(n,t)}$ be linear star configurations in $\P^n$ of type $(n,t)$ and $(n,s)$ with  $s\ge t\ge n$, respectively. Then an Artinian ring $R/(I_{\X}+I_{\Y})$ has the weak-Lefschetz property. 
\end{thm} 

%
%

We also recall the following remark in \cite{AS:1}.

\begin{rem}[{\cite[Remark 4.3]{AS:1}}] If $\X^{(n,s)}$ and $\X^{(n,t)}$ are not linear star-configurations  in Theorem~\ref{T:204-20100706}, then Theorem~\ref{T:204-20100706} may not hold in general. 
For example, assume that $\X:=\X^{(2,4)}$ and $\Y:=\X^{(2,4)}$ are star-configurations in $\P^2$ of type $(2,4)$ defined by general forms in $R=\k[x_0,x_1,x_2]$ of degree $2$. Then, by  Theorem~\ref{T:20140307-303}, the Hilbert functions of $R/I_\X$ and $R/I_\Y$  are
$$
\begin{matrix}
1 & 3 & 6 & 10 & 15 & 21 & 24 & \ra,
\end{matrix}
$$
and thus
$$
\sigma(\X)=\sigma(\Y)=7. 
$$
Furthermore, the Hilbert function of $R/I_{\X\cup\Y}$, obtained by CoCoA, is
$$
\begin{matrix}
1 & 3 & 6 & 10 & 15 & 21 & 28 & 36 & 45 & 48 & \ra,
\end{matrix}
$$
and thus
$$
\begin{array}{llllllllllllllllllll}
\H(R/I_\X+I_\Y,6)
& = & \H(R/I_\X,6)+\H(R/I_\Y,6)-\H(R/I_{\X\cup\Y},6) \\[.5ex]
& = & 24+24-28 \\
& = & 20 \\
& \ne & \H(R/I_\X,6).
\end{array}
$$
This does not satisfy the conditions in Proposition~\ref{P:202-20100706}, and thus we do not know if Theorem~\ref{T:204-20100706} still holds for this case when $\X$ and $\Y$ are  star-configurations in $\P^n$ defined by general forms of degree $d$ with $d\ge 2$. 
\end{rem} 

Theorem~\ref{T:204-20100706}  gives   a  complete answer to Question~\ref{Q:20131219-102} for  $d=1$. In other words, Question~\ref{Q:20110219-102} for $d>1$ is still open.  Thus, we restate Question~\ref{Q:20110219-102} as follows.

\begin{ques}[Restated Question~\ref{Q:20131219-102}] Let $\X:=\X^{(n,s)}$ and $\Y:=\X^{(n,t)}$ be star-configurations in $\P^n$ of type $(n,s)$ and $(n,t)$, respectively, defined by general forms of degree $d>1$. Does an Artinian ring $R/(I_{\X}+I_{\Y})$ have the weak-Lefschetz property? 
\end{ques} 

Furthermore, we have the following question in general.

\begin{ques} Let $\X:=\X^{(n,s)}$ and $\Y:=\X^{(n,t)}$ be star-configurations in $\P^n$ of type $(n,s)$ and $(n,t)$ defined by general forms of degrees $d_1,\dots,d_s$ and $d_1',\dots,d_t'$, respectively. Does an Artinian ring $R/(I_{\X}+I_{\Y})$ have the weak-Lefschetz property? 
\end{ques}

Now we move on to a more  general case of the union of two star-configurations in $\P^2$ of type $(2,s)$ with $s\ge 2$. 
In \cite{S:2}, the author  found that if $\X$ and $\Y$ are star-configurations in $\P^2$ defined by general forms of degrees $\le 2$ and $\sigma(\X)\ne \sigma(\Y)$, then $R/(I_\X+I_\Y)$ has the weak-Lefschetz property.

%

\begin{thm}[{\cite[Theorem 3.3]{S:2}}]\label{T:20110720-303} Let $\X:=\X^{(2,s)}$ and $\Y:=\X^{(2,t)}$ be star-configurations in $\P^2$ of type $(2,s)$ and $(2,t)$ defined by general forms $F_1,\dots,F_s$ and $G_1,\dots,G_t$ in $R=\k[x_0,x_1,x_2]$, respectively, with $s,t\ge 3$. 
Assume $\deg(F_i)\le 2$ for $i=1,\dots,s$ and $\deg(G_j)\le 2$ for $j=1,\dots,t$. If $\sigma(\X)\ne \sigma(\Y)$, then an Artinian ring $R/(I_{\X}+I_{\Y})$ has the weak-Lefschetz property. 
\end{thm} 

%
%

Before we introduce a new Artinian quotient of a coordinate ring of a star-configuration in $\P^2$ having the weak- Lefschetz property without the condition $\sigma(\X)\ne \sigma(\Y)$, we need the following two propositions. 

\begin{pro} [{\cite[Proposition 3.6]{S:2}}]\label{P:20111004-309} Let $\X^{(2,s)}$ be a star-configuration in $\P^2$ defined by general forms $F_1,\dots,F_s$ of degrees $1\le d_1\le \cdots \le d_s$ with $s\ge 3$. Then
$$
\begin{matrix} 
\sigma(\X^{(2,s)})=\Big[{\ds\sum}_{i=1}^{s} d_i\Big]-1. 
\end{matrix}
$$
\end{pro}

\begin{pro} [{\cite[Proposition 2.6]{S:2}}]\label{P:20110720-205} If $\X^{(2,s)}$ is a star-configuration in $\P^2$ of type $(2,s)$ defined by general forms $F_1,\dots,F_s$ of degrees $1\le d_1\le \cdots\le d_s\le 2$ with $s\ge 3$, then $\X^{(2,s)}$ has generic Hilbert function. In particular,  the Hilbert function of $\X^{(2,s)}$ is
$$
\begin{array}{llllllllllllllllllllllllll}
\H_\X & : & 1 & \ds\binom{1+2}{2} & \cdots & \ds\binom{2+(d-3)}{2} & \deg(\X^{(2,s)}) & \ra,
\end{array}
$$
where $d=\sum_{j=1}^s d_j$.
\end{pro}

We now discuss an Artinian quotient of a coordinate ring of a star-configuration in $\P^2$ having the Weak Lefschetz property without the condition $\sigma(\X)\ne \sigma(\Y)$.

\begin{pro}\label{P:20140322-411} Let $\X:=\X^{(2,s)}$ and $\Y:=\X^{(2,t)}$ be as in Theorem~\ref{T:20110720-303}. Assume $\deg(F_i)=1$ for every $1\le i \le s$ and $\deg(G_j)\le 2$ for every $1\le j \le t$. Then $R/(I_{\X}+I_\Y)$ has the weak-Lefschetz property. 
\end{pro} 

\begin{proof} If $\sigma(\X)\ne \sigma(\Y)$, then it is immediate from Theorem~\ref{T:204-20100706}.  Now assume that $\sigma(\X)=\sigma(\Y)$, i.e., 
$
\begin{matrix}
\sum^s_{i=1} \deg(F_i)= \sum^t_{i=1} \deg(G_i).
\end{matrix}
$
It is from Proposition~\ref{P:20110720-205} 
that the Hilbert functions of  $\X$ and $\Y$ are 
$$
\begin{array}{lllllllllllllllllllllllll}
\H_\X & : & 1 &\ds \binom{1+2}{2} & \cdots & \ds\binom{(s-3)+2}{2} & \ds\overset{(s-2)\text{-nd}}{\deg(\X)} & \ra, & \text{and} \\
\H_\Y & : & 1 & \ds\binom{1+2}{2} & \cdots & \ds\binom{(s-3)+2}{2} & \ds\overset{(s-2)\text{-nd}}{\deg(\Y)} & \ra.
\end{array} 
$$
Note that
$$
\H_\X(s-2)=\deg(\X)=\binom{2+(s-2)}{2}.
$$
In other words, $I_\X$ has no generators in degree $s-2$, and thus
$$
\H(R/(I_\X+I_\Y),s-2)=\H_\Y(s-2).
$$
Moreover, since 
$$
\H_\X(i)=\H_\Y(i)=\binom{2+i}{2}
$$
for $0\le i \le s-3$, and so
$$
\H(R/(I_\X+I_\Y),i)=\H_\Y(i)=\binom{2+i}{2}.
$$
Hence we get that
$$
\H(R/(I_\X+I_\Y),i)=\H(R/I_\Y,i)
$$
for every $0\le i\le s-2=\sigma(\Y)-1$. Furthermore, since
$$
R/(I_\X+I_\Y)\simeq (R/I_\Y)/((I_\X+I_\Y)/I_\Y)
$$
is an Artinian quotient of the coordinate ring $R/I_\Y$, by Proposition~\ref{P:202-20100706} $R/(I_\X+I_\Y)$ has the weak Lefschetz property, as we wished.  
\end{proof}  

By Corollary~\ref{C:20130107-203}, we often use a product of linear forms $L_1,\dots,L_d$ in $R=\k[x_0,\dots,x_n]$ instead of  general forms $F$ of degree $d$ to construct a star-configuration in $\P^n$ for this section. In \cite{S:2}, the author showed that if $\X:=\X^{(2,s)}$ is a star-configuration in $\P^2$ of type $(2,s)$ defined by general quadratic forms $F_1,\dots,F_s$ and $\Y:=\X^{(2,s+1)}$ is a star-configuration in $\P^2$ of type $(2,s+1)$ defined by general quadratic forms $G_1,\dots,G_s$ and a general linear form $L$, then an Artinian ring $R/(I_\X+I_\Y)$ has the weak-Lefschetz property with a Lefschetz element $L$ (see \cite[Theorem 3.7] {S:2}). We shall generalize this result with the condition $\deg(F_i)=\deg(G_i)\le 2$ for every $i=1,\dots,s$. 

For the rest of this section, to distinguish two star-configurations, we shall use the following notations and symbols for lines and points in pictorial description.
$$\begin{array}{lllllllllllllllllllll}
\text{a solid line } & \begin{picture}(325,0)(35,26)
\psline(1,1)(2,1)  \end{picture} \hskip -10.8 true cm & \L_i & 
\text{is a line defined by a linear form} & L_i,  \text{ and}\\
\text{a dashed line } & \begin{picture}(325,0)(35,26)
\psline[linestyle=dashed](1,1)(2,1)  \end{picture} \hskip -10.8 true cm & \M_i & 
\text{is a line defined by a linear form} & M_i
\end{array}
$$
for $1\le i\le s$ with $s\ge 2$. We also define that 
$$
\begin{array}{lllllllllllllllllll}
P_{i,j} & \text{is a point defined by linear forms }  L_i, L_j , \text{ and}\\
Q_{i,j} & \text{is a point defined by linear forms } M_i, M_j
\end{array}
$$
where $L_i,L_j$ and $M_i,M_j$ are linear forms in $R$ for $1\le i < j\le s$ with $s\ge 2$.

%

\begin{lem}\label{L:20111214-402}  Let $\X$ be the union of two star-configurations $\X_1:=\X^{(2,3)}$ and $\X_2:=\X^{(2,2)}$ in $\P^2$ of type $(2,3)$ and $(2,2)$, respectively. Then $\X$ has  generic Hilbert function
$$
\begin{matrix}
1 & 3 & 6 & 10 & 15 & 16 & \ra .
\end{matrix} 
$$
\end{lem} 

\begin{proof} Without loss of generality, we assume that $\X_1$ is defined by quadratic forms $L_1L_2$, $L_3L_4$, and $L_5L_6$, where $L_i$ is a linear form defining a line $\L_i$  for $i=1,\dots,6$, and that $\X_2$ is defined by quadratic forms $M_1M_2$ and $M_3M_4$, where $M_i$ is a linear form defining a line $\M_i$ for $i=1,\dots,4$.  Furthermore, we assume that $L_1$ vanishes on four points in $\X_1$ and one point  defined by two linear forms $M_1$ and $M_4$ in $\X_2$, and that $M_2$ vanishes on two points in $\X_2$ and one point  defined by linear forms $L_3$ and $L_6$ in $\X_1$ (see Figure~\ref{FIG:20111214-001} again). 

\begin{figure}[ht] 
\vskip .3pc
\begin{picture}(325, 90)(-50,5)

\psline(1.82,-.2)(4.7,3.3)  

\put(136,95){$\L_{1}$} 

\psline(2.1,-.4)(4.7,2.8)  

\put(136,78){$\L_{2}$}

\psline(2.3,1)(6.3,1)  

\psline(2.3,1.3)(6.3,1.3)  

\put(182,35){$\L_{5}$}

\put(182,25){$\L_{6}$} 

\psline(4,3.3)(6,.5)  

\psline(4,2.7)(6,-0.1) 

\put(100,95){$\L_{3}$}

\put(100,78){$\L_{4}$}

\psline[linestyle=dashed](1.7,0.05)(3.3,0.59) 

\put(95,15){$\M_1$}

\psline[linestyle=dashed](2,-.3)(6.2,1.2) 

\put(95,-5){$\M_2$}

\psline[linestyle=dashed](2.1,.7)(3,-.2) 

\put(50,23){$\M_3$}

\psline[linestyle=dashed](1.88,0.5)(2.77,-.4) 

\put(38,13){$\M_4$}

\end{picture}
\vskip .9pc
\caption{} \label{FIG:20111214-001}
\end{figure}

Let $N\in (I_\X)_4,$   then by {\em Bez\'{o}ut}'s Theorem, 
$$
N=\alpha L_1L_2M_2L_5
$$
for some $\alpha\in \k$. Moreover, since $N$ has to vanish on two more points $P_{4,6},Q_{1,3}$ in $\X$, where none of $L_1$, $L_2$, $M_2$, and $L_5$ can vanish, we get that $N=0$. Therefore, the Hilbert function of $\X$ is 
$$
\begin{matrix}
1 & 3 & 6 & 10 & 15 & 16 & \ra ,
\end{matrix} 
$$
as we wished. 
\end{proof}

\begin{thm} \label{T:20111106-401} Let $\X$  be the union of two star-configurations $\X_1:=\X^{(2,3)}$ and $\X_2:=\X^{(2,3)}$ in $\P^2$ of type $(2,3)$. Then $\X$ has  generic Hilbert function
$$
\begin{matrix}
1 & 3 & 6 & 10 & 15 & 21 & 24 & \ra .
\end{matrix} 
$$
\end{thm}

\begin{proof} 
Without loss of generality, we assume that $\X_1$ and $\X_2$ are defined by quadratic forms $L_iL_{i+1}$ and $M_iM_{i+1}$,   respectively,  for $i=1,3,5$. (see Figure~\ref{FIG:20111224-002}). We also assume that
$$
\begin{array}{llllllllllllll}
\text{a linear form $L_1$ vanishes on $6$ points } P_{1,3}, P_{1,4}, P_{1,5}, P_{1,6},   Q_{1,6},Q_{2,5}, \text{ and}\\
\text{a linear form $L_2$ vanishes on $5$ points } P_{2,3}, P_{2,4}, P_{2,5}, P_{2,6},   Q_{2,6}.
\end{array}
$$

\begin{figure}[ht] 
\vskip .3pc
\begin{picture}(325, 90)(-50,5)

\psline(1.82,-.2)(4.7,3.3)  

\put(136,95){$\L_{1}$} 

\psline(2.1,-.4)(4.7,2.8)  

\put(136,78){$\L_{2}$}

\psline(2.3,1)(6.3,1)  

\psline(2.3,1.3)(6.3,1.3)  

\put(182,35){$\L_{5}$}

\put(182,25){$\L_{6}$} 

\psline(4,3.3)(6,.5)  

\psline(4,2.7)(6,-0.1) 

\put(100,95){$\L_{3}$}

\put(100,78){$\L_{4}$}

\psline[linestyle=dashed](1.4,0.55)(4.2,0.55) 

\put(122,12){$\M_1$}

\psline[linestyle=dashed](1.5,0.2)(4.1,0.23) 

\put(120,3){$\M_2$}

\psline[linestyle=dashed](2.33,.7)(3.68,-1.7) 

\put(60,-58){$\M_3$}

\psline[linestyle=dashed](1.88,0.7)(3.22,-1.7) 

\put(75,-58){$\M_4$}

\psline[linestyle=dashed](3.5,0.9)(2.5,-1.7) 

\put(90,-58){$\M_5$}

\psline[linestyle=dashed](3.9,0.9)(2.9,-1.7) 

\put(105,-58){$\M_6$}
\end{picture}
\vskip 4.5pc
\caption{} \label{FIG:20111224-002}
\end{figure}

For every $N\in R_5$, by {\em Bez\'{o}ut}'s theorem, 
$$
N=\alpha L_1L_2M_3M_4L_5,
$$
for some $\alpha \in \k$. Moreover, since $N$ has to vanish on a point $P_{3,6}$, where none of $L_1,L_2,M_3,M_4$, and $L_5$  vanishes, we have $N=0$. Hence the Hilbert function of $\X$ is of the form
\begin{equation}\label{EQ:20111225-404} 
\begin{array}{llllllllllll}
\H_\X & : & 1 & 3 & 6 & 10 & 15 & 21 & \cdots.
\end{array}
\end{equation}

Let $\Y_1:=\X_1$, and $\Y_2:=\X^{(2,2)}$ be a star-configuration in $\P^2$ of type $(2,2)$ defined by quadratic forms $M_1M_2$ and $M_3M_4$, respectively. Define $\Y:=\Y_1\cup\Y_2$.  By Lemma~\ref{L:20111214-402} the Hilbert function of $\Y$ is
\begin{equation}\label{EQ:20111225-405} 
\begin{array}{llllllllllll}
\H_\Y & : & 1 & 3 & 6 & 10 & 15 & 16 & \ra.  
\end{array}
\end{equation} 

Let $\Z_1:=\X_1$ and $\Z_2$ be a star-configuration in $\P^2$ of type $(2,3)$ defined by $M_1M_2,M_3M_4$, and $M_5$. Define $\Z:=\Z_1\cup\Z_2$ (see Figure~\ref{FIG:20111224-002}), and  let  $G_4=M_1\cdots M_4$.

%
%
%
%
%
%
%
%
%
%
%
%
%
%
%
%
%
%
%
%
%
%
%
%

Using equation~\eqref{EQ:20111225-405} and the exact sequence
$$
\begin{matrix}
0 & \ra & R/I_\Z & \ra & R/I_\Y \bigoplus R/(M_5,G) & \ra & R/(I_\Y,M_5,G) & \ra & 0, 
\end{matrix}
$$
we get that
$$
\begin{array}{rcccccccclrrr}
\H(R/I_\Z,-)    & : & 1 & 3 & 6 & 10 & 15 & - & 20 & \ra, \\
\H(R/I_\Y,-)    & : & 1 & 3 & 6 & 10 & 15 & 16     & 16 & \ra, \\
\H(R/(M_5,G_4),-) & : & 1 & 2 & 3 &  4 & 4  & 4      &  4 & \ra, \\
\H(R/(I_\Y,M_5,G_4),-) 
                & : & 1 & 2 & 3 &  4 &  4 & -  & 0  & \ra, \\
\H(R/(I_\Y,M_5),-) 
                & : & 1 & 2 & 3 &  4 &  5 & 1      & 0  & \ra \kern-.2em{.}
\end{array} 
$$
Moreover, since $\H(R/(I_\Y,M_5),5)=1$, we see that
$$
\H(R/I_\Y,M_5,G_4),5)=0, \text{ or \,} 1.
$$
In other words, there are only two possible Hilbert functions for $R/I_\Z$:
$$
\begin{array}{lllllllllllllllllllllll}
(1) & : & 1 & 3 & 6 & 10 & 15 & 20 & 20 & \ra, & \text{or} \\
(2) & : & 1 & 3 & 6 & 10 & 15 & 19 & 20 & \ra.
\end{array} 
$$

However, using the same argument as in the proof of Theorem 3.3 in \cite{S:1}, one can show that (1) is the Hilbert function of $\Z$. 

Recall that  $\X_1$ and $\X_2$ are star-configurations in $\P^2$ of type $(2,3)$  defined by quadratic forms $L_iL_{i+1}$ and $M_iM_{i+1}$,  respectively,   for $i=1,3,5$. Let  $\X:=\X_1\cup\X_2$ and $G_4=M_1\cdots M_4$. 

By equation~\eqref{EQ:20111225-404} and the following exact sequence
$$
\begin{matrix}
0 & \ra & R/I_\X & \ra & R/I_\Z \bigoplus R/(M_6,G) & \ra & R/(I_\Z,M_6,G) & \ra & 0, 
\end{matrix}
$$
we get that
$$
\begin{array}{rccccccccclrrr}
\H(R/I_\X,-)    & : & 1 & 3 & 6 & 10 & 15 & 21 & - & 24 & \ra, \\
\H(R/I_\Z,-)    & : & 1 & 3 & 6 & 10 & 15 & 20 & 20     & 20 & \ra, \\
\H(R/(M_6,G_4),-) & : & 1 & 2 & 3 &  4 & 4  & 4  &  4     &  4 & \ra, \\
\H(R/(I_\Z,M_6,G_4),-) 
                & : & 1 & 2 & 3 &  4 &  4 & 3  & -  & 0  & \ra, \\
\H(R/(I_\Z,M_6),-) 
                & : & 1 & 2 & 3 &  4 &  5 & 5  & 0      & 0  & \ra\kern-.25em{.}
\end{array} 
$$
Moreover, since $\H(R/(I_\Z,M_6),6)=0$, we see that $\H(R/I_\Z,M_6,G_4),6)=0$ as well. Thus the Hilbert function of $\X$ is
$$
\begin{array}{rccccccccclrrr}
\H(R/I_\X,-)    & : & 1 & 3 & 6 & 10 & 15 & 21 & 24 & \ra, 
\end{array} 
$$
as desired. 
\end{proof}

By {\em Bez\'{o}ut}'s Theorem and  Theorem~\ref{T:20111106-401}, we obtain the following proposition by double induction on $d$ and $s$
 and thus omit the proof. 

\begin{pro} \label{P:20111231-501} Let $\X$ be the union of two star-configurations in $\P^2$ of type $(2,s)$ defined by general forms in $R=\k[x_0,x_1,x_2]$ of degree $d$ with $s\ge 4$ and $d\ge 2$. Then
$$
(I_\X)_{ds}=\{0\}.
$$ 
\end{pro}



\begin{lem}\label{L:20121103-309} Let $\X$ and $\Y$ be star-configurations in $\P^2$ of type $(2,s)$ defined by general forms $F_1,\dots,F_s$ and  $G_1,\dots,G_s$ with $s\ge 3$, respectively. Assume that
$$
\deg(F_i)=\deg(G_i)=
\begin{cases}
1, & \text{for } i=1,\dots,\ell, \\
2, & \text{for } i=\ell+1,\dots,s,
\end{cases}
$$
with $0\le \ell<s$. Then 
$$
\begin{array}{lllllllllllllllllllllll} 
\dim _\k (I_{\X}+I_{\Y})_{2s-\ell-2}= 2(s-\ell).
\end{array} 
$$
\end{lem}

\begin{proof} We shall prove this by induction on $s$.  If $s\ge 3$ and $\ell=0$, then by Propositions~\ref{T:20111106-401}  and \ref{P:20111231-501}, and the exact sequence 
\begin{equation} \label{EQ:20121103-301} 
\begin{matrix}
0 & \ra & R/I_{\X\cup\Y} & \ra & R/I_\X \oplus R/I_\Y & \ra & R/(I_\X+I_\Y) & \ra & 0,
\end{matrix}
\end{equation}
we have
$$
\dim_\k(I_\X+I_\Y)_{2s-\ell-2}=\dim_\k(I_\X+I_\Y)_{2s-2}=2s=2(s-\ell).           
$$

Assume $1\le \ell <s$. If $s=3$, then using the same idea as in the proof of Theorem~\ref{T:20111106-401}, one can prove that
$$
\dim_\k(I_{\X\cup\Y})_{4-\ell}=\dim_\k(I_{\X\cup\Y})_{2s-\ell-2}=0,
$$
and thus 
$$
\dim_\k(I_\X+I_\Y)_{2s-\ell-2}=\dim_\k (I_\X+I_\Y)_{4-\ell}=2(3-\ell)=2(s-\ell).
$$
Now suppose $s>3$. 
Let $\FF_1$ be a line defined by a linear form $F_1$ and $N\in (I_{\X\cup\Y})_{2s-\ell-2}$. Notice that $F_1$ vanishes on $(2s-\ell-1)$-points on a line $\FF_1$. By {\em Bez\'{o}ut}'s Theorem, we have
$$
N=F_1 N'
$$
for some $N'\in R_{2s-\ell-3}$. Let $\X'$ and $\Y'$ be star-configurations in $\P^2$ of type $(2,s-1)$ defined by $F_2,\dots,F_s$ and $G_2,\dots,G_s$, respectively. Then, 
$$
\begin{array}{llllllllllllllll}
N' & \in & (I_{\X'\cup\Y'})_{2s-\ell-3} \\
   & =   & (I_{\X'\cup\Y'})_{2(s-1)-(\ell-1)-2} \\
   & =   & \{0\}, \quad (\text{by induction on $s$})
\end{array} 
$$
and thus $(I_{\X\cup\Y})_{2s-\ell-2}=\{0\}$ as well. Therefore, using equation~\eqref{EQ:20121103-301}, we get that
$$
\dim_\k (I_\X+I_\Y)_{2s-\ell-2}=2(s-\ell),
$$
which completes the proof of this lemma. 
\end{proof} 

\begin{rem} \label{R:20121110-308} Let $\X$ and $\Y$ be as in Lemma~\ref{L:20121103-309}. Since $(I_{\X\cup\Y})_{2s-\ell-2}=\langle \tilde F_{\ell+1},\dots,\tilde F_s, \tilde G_{\ell+1},\dots,\tilde G_s \rangle$ and by Lemma~\ref{L:20121103-309}, $\dim_\k (I_{\X\cup\Y})_{2s-\ell-2}=2(s-\ell)$, we see that the following $2(s-\ell)$-forms 
$$
 \tilde F_{\ell+1},\dots,\tilde F_s, \tilde G_{\ell+1},\dots,\tilde G_s 
$$
are linearly independent. 
\end{rem} 

We are now ready to prove the following theorem, which generalises  Theorem 3.7 in \cite{S:2}. 

\begin{thm}\label{T:20121102-208} Let $\X$  be a star-configuration in $\P^2$ of type $(2,s)$ defined by general forms $F_1,\dots,F_s$ and $\Y$ be a star-configuration in $\P^2$ of type $(2,s+1)$ defined by general forms $G_1,\dots,G_s$ and a general linear form $L$. Assume that 
$$
\deg(F_i)=\deg(G_i)=
\begin{cases}
1, & \text{for } i=1,\dots,\ell, \\
2, & \text{for } i=\ell+1,\dots,s,
\end{cases}
$$
where $0\le \ell <s$. Then an Artinian ring $R/(I_\X+I_\Y)$ has the weak-Lefschetz property with a Lefschetz element $L$. 
\end{thm}

\begin{proof}
First, by Proposition~\ref{P:20111004-309}, we have $\sigma(\X)<\sigma(\Y)$, and hence
by Theorem~\ref{T:20110720-303},     $R/(I_\X+I_\Y)$ has the weak-Lefschetz property. It suffices to show that $L$ is a Lefschetz element. 

Note that the Hilbert function of $R/(I_\X+I_\Y)$ is of the form:
\begin{equation}\label{EQ:20110829-305}
\begin{array}{llllllllllllllllllllll}
\H_{R/(I_\X+I_\Y)}(-) & : & 1 & \ds\binom{1+2}{2}   & \cdots & \ds\binom{(2s-\ell-3)+2}{2}  & \ds\overset{(2s-\ell-2)\text{-nd}}{\binom{(2s-\ell-2)+2}{2}-(s-\ell)} & \cdots . 
\end{array}
\end{equation}

We shall show that
$$
\begin{matrix} 
\dim_\k (I_\X+I_\Y)_{2s-\ell-1}=4s-3\ell, \quad {\text{i.e.,}} \quad 
\H_{R/(I_\X+I_\Y)}(2s-\ell-1) =\binom{(2s-\ell-1)+2}{2}-(4s-3\ell).
\end{matrix} 
$$
Consider the following $(4s-3\ell)$-forms in $(I_\X+I_\Y)_{2s-\ell-1}$
$$
\begin{array}{lllllllllllllllllll}
\tilde F_1, \cdots,\tilde F_\ell, x_0\tilde F_{\ell+1},x_1\tilde F_{\ell+1}, x_2\tilde F_{\ell+1} ,\dots, x_0\tilde F_s,x_1\tilde F_s, x_2\tilde F_s, 
 L\tilde G_{\ell+1}, \dots,L\tilde G_s,
\end{array} 
$$
where
$$
\begin{array}{llllllllllllll}
\tilde F_i=\frac{\prod^s_{j=1} F_j}{F_i} \text{ and } 
\tilde G_i=\frac{\prod^s_{j=1} G_j}{G_i},
\end{array} 
$$
for $i=1,\dots,s$. 

Suppose that
\begin{equation}\label{EQ:20110829-306}
\begin{array}{lllllllllllllllllllll}
\alpha_1 \tilde F_1+ \cdots+\alpha_\ell \tilde F_\ell+
\alpha_{0\ell+1}x_0\tilde F_{\ell+1} +\alpha_{1\ell+1}x_1\tilde F_{\ell+1}+\alpha_{2\ell+1}x_2\tilde F_{\ell+1}+\cdots+ \\
\alpha_{0s}x_0\tilde F_s +\alpha_{1s}x_1\tilde F_s+\alpha_{2s}x_2\tilde F_s+ \beta_{\ell+1} L\tilde G_{\ell+1}+\cdots+\beta_s L \tilde G_s =0,
\end{array}
\end{equation}
where $\alpha_i, \alpha_{ij}, \beta_i\in \k$ for every $i,j$. 
Since $L$ is a linear factor of the form
$$
\alpha_1 \tilde F_1+ \cdots+\alpha_\ell \tilde F_\ell+
\alpha_{0\ell+1}x_0\tilde F_{\ell+1} +\alpha_{1\ell+1}x_1\tilde F_{\ell+1}+\alpha_{2\ell+1}x_2\tilde F_{\ell+1}+\cdots+ 
\alpha_{0s}x_0\tilde F_s +\alpha_{1s}x_1\tilde F_s+\alpha_{2s}x_2\tilde F_s,
$$
and $L$ is a non-zero divisor of $R/I_\X$, we see that
$$
\beta_{\ell+1}\tilde G_{\ell+1}+\cdots+\beta_s\tilde G_s\in (I_\X)_{2s-\ell-2}.
$$
Moreover,  by Lemma~\ref{L:20121103-309} (see also Remark~\ref{R:20121110-308})
$$
\beta_{\ell+1}=\cdots=\beta_s=0.
$$
Thus we rewrite equation~\eqref{EQ:20110829-306} as 
\begin{equation}\label{EQ:20110829-307}
\begin{array}{lllllllllllllllllllll}
\alpha_1 \tilde F_1+ \cdots+\alpha_\ell \tilde F_\ell+
\alpha_{0\ell+1}x_0\tilde F_{\ell+1} +\alpha_{1\ell+1}x_1\tilde F_{\ell+1}+\alpha_{2\ell+1}x_2\tilde F_{\ell+1}+\cdots+ \\
\alpha_{0s}x_0\tilde F_s +\alpha_{1s}x_1\tilde F_s+\alpha_{2s}x_2\tilde F_s
=0.
\end{array}
\end{equation}
Note that $F_j\mid \tilde F_i$ for every $j\ne i$. Hence
$$
F_1 \mid \alpha_1 \tilde F_1,
$$
and so $\alpha_1=0$. By the same method as above, one can show that $\alpha_1=\cdots=\alpha_\ell=0$. Moreover, since 
$$
F_{\ell+1} \mid (\alpha_{0\ell+1}x_0\tilde F_{\ell+1} +\alpha_{1\ell+1}x_1\tilde F_{\ell+1}+\alpha_{2\ell+1}x_2\tilde F_{\ell+1}),
$$
and $F_{\ell+1}\nmid \tilde F_{\ell+1}$, we get that 
$$
F_{\ell+1} \mid (\alpha_{0\ell+1}x_0 +\alpha_{1\ell+1}x_1+\alpha_{2\ell+1}x_2). 
$$
This implies that 
$$
\alpha_{0\ell+1}=\alpha_{1\ell+1}=\alpha_{2\ell+1}=0.
$$
By the same idea as above, one can show that
$$
\alpha_{ij}=0
$$
for every $0\le i\le 2$ and $\ell+1\le j\le s$, and thus
$$
\dim_\k (I_\X+I_\Y)_{2s-\ell-1}=4s-3\ell.
$$
In other words, the Hilbert function of $R/(I_\X+I_\Y)$ in degree $2s-\ell-1$ is 
\begin{equation}\label{EQ:20110829-308}
\begin{matrix} 
\H_{R/(I_\X+I_\Y)}(2s-\ell-1)={\binom{(2s-\ell-1)+2}{2}-(4s-3\ell)}.
\end{matrix} 
\end{equation}

\medskip

Now we shall show that $L$ is a Lefschetz element of $R/(I_\X+I_\Y)$. Note that
$$
\begin{array}{llllllllllllllllllll}
I_\X+I_\Y
& = & (\tilde F_1,\dots,\tilde F_s, L\tilde G_1,\dots, L\tilde G_s, \prod^s_{i=1}G_i). 
\end{array} 
$$
Consider a multiplication map by $L$
\begin{equation}\label{EQ:20110726-002} 
\times L : (R/(I_\X+I_\Y))_{2s-\ell-2} \ra (R/(I_\X+I_\Y))_{2s-\ell-1}
\end{equation} 
and let $N \in {\rm Ker}(\times L)$. Then
$$
\begin{array}{llllllllllllllll}
N\cdot L
& \in & (I_\X+I_\Y)_{2s-\ell-1} \\
& = & \langle
\tilde F_1, \cdots,\tilde F_\ell, x_0\tilde F_{\ell+1},x_1\tilde F_{\ell+1}, x_2\tilde F_{\ell+1} ,\dots, x_0\tilde F_s,x_1\tilde F_s, x_2\tilde F_s, 
 L\tilde G_{\ell+1}, \dots,L\tilde G_s\rangle.
\end{array} 
$$
This implies that for some $\alpha_1,\dots,\alpha_\ell \in \k$, $K_{\ell+1},\dots,K_s\in R_1$, and $\beta_{\ell+1},\dots,\beta_s\in \k$, 
$$
\begin{array}{llllllllllllllllllll}
N\cdot L 
& = & {\ds \sum}^\ell_{i=1} \alpha_i \tilde F_i +{\ds\sum}^s_{i=\ell+1} K_i  \tilde F_i +{\ds \sum}^s_{i=\ell+1} \beta_i L \tilde G_i ,
\end{array} 
$$
i.e.,
$$
\begin{matrix} 
\Big(N-{\ds\sum}^s_{i=\ell+1} \alpha_i \tilde G_i\Big)\cdot L  
 =  {\ds\sum}^\ell_{i=1} \alpha_i \tilde F_i +{\ds\sum}^s_{i=\ell+1} K_i  \tilde F_i
 \in  I_\X.
 \end{matrix} 
$$
In other words,  $N-\sum^s_{i=\ell+1} \alpha_i \tilde G_i$ is contained in the kernel of the multiplication map by $L$
$$
\times L : (R/I_\X)_{2s-\ell-2} \ra (R/I_\X)_{2s-\ell-1}.
$$
Since $L$ is a general linear form in $R_1$, this multiplication map by $\times L$
$$
(R/I_\X)_{i} \ra (R/I_\X)_{i+1}
$$
is injective for all $i\ge 0$. It  follows that
$$
N-\sum^s_{i=\ell+1} \alpha_i \tilde G_i\in (I_\X)_{2s-\ell-2},
$$
and so
$$
N\in (\tilde F_{\ell+1},\dots,\tilde F_s,\tilde G_{\ell+1},\dots,\tilde G_s)_{2s-\ell-2}. 
$$
Moreover, since $\tilde F_{\ell+1}, \dots, \tilde F_s \in (I_\X+I_\Y)_{2s-\ell-2}$, but $\tilde G_{\ell+1},\dots, \tilde G_s \notin (I_\X+I_\Y)_{2s-\ell-2}$, we get that
$$
\dim_\k [{\rm Ker}(\times L)]_{2s-\ell-2}=s-\ell.
$$
Hence
$$
\begin{array}{lllllllllllllllllll}
\dim_\k [{\rm Im} (\times L)]_{2s-\ell-1}
& = & \H_{R/(I_\X+I_\Y)}(2s-\ell-2)-\dim_\k [{\rm Ker}(\times L)]_{2s-\ell-2} \\[.3ex]
& = & \left[\binom{2s-\ell-2)+2}{2}-(s-\ell)\right]-(s-\ell) \\[.5ex]
& = & \binom{(2s-\ell-1)+2}{2}-(2s-\ell)-2(s-\ell) \\[.5ex]
& = & \binom{(2s-\ell-1)+2}{2}-(4s-3\ell) \\[.5ex]
& = & \H_{R/(I_\X+I_\Y)}(2s-\ell-1), \qquad   (\text{by equation~\eqref{EQ:20110829-308}}).
\end{array}
$$
This indicates that  the multiplication map by $L$ in equation~\eqref{EQ:20110726-002} is surjective. 

\medskip

Now consider the following exact sequence:
$$
\begin{matrix}
0 & \ra & ((I_\X+I_\Y:L)/(I_\X+I_\Y))_{d-1} & \ra & A_{d-1} & \overset{\times L}{\ra} & A_{d} & \ra & (A/LA)_d & \ra & 0.
\end{matrix}
$$
Since the multiplication map by $L$  is surjective in degree $2s-\ell-2$, we have 
$$
\dim_\k(A/LA)_d=0
$$
for $d\ge 2s-\ell-1$. This means  the multiplication map by $L$ is surjective in degrees $d\ge 2s-\ell-2$. Therefore, $L$ is a Lefschetz element of $A$. This completes the proof. 
\end{proof}

\section*{Acknowdegement}
The authors would like to express their sincere appreciation to Professor Anthony V. Geramita and Professor Jeaman Ahn for valuable discussions.

\end{document}